\documentclass[a4paper,12pt]{article}
\usepackage[english]{babel}
\usepackage[utf8]{inputenc}
\usepackage{paralist,url,verbatim, anysize}
\usepackage{amscd}
\usepackage[all,cmtip]{xy}
\usepackage{amsmath,amsthm,amssymb,amsfonts}
\usepackage[bookmarksopen=true]{hyperref}
\usepackage[usenames, dvipsnames]{color}
\usepackage{graphicx}

\theoremstyle{plain}
\newtheorem{theorem}{\bf Theorem}[section]

\newtheorem{corollary}[theorem]{Corollary}
\newtheorem{lemma}[theorem]{Lemma}
\newtheorem{proposition}[theorem]{Proposition}

\newtheorem{question}[theorem]{Question}
\newtheorem*{theorem*}{Main Theorem}
\newtheorem*{theorem**}{Theorem}

\theoremstyle{definition}
\newtheorem{remark}[theorem]{Remark}
\newtheorem*{remark*}{Remark}
\newtheorem{definition}[theorem]{Definition}
\newtheorem*{definition*}{Definition}

\newtheorem{example}{\bf Example}

\newtheorem*{conjecture*}{Conjecture}

\newcommand{\NN}{\mathbb{N}}
\newcommand{\ZZ}{\mathbb{Z}}
\newcommand{\QQ}{\mathbb{Q}}
\newcommand{\PP}{\mathbb{P}}
\newcommand{\C}{\mathcal{C}}
\newcommand{\mm}{\mathfrak{m}}

\newcommand{\qq}{\mathfrak{q}}
\newcommand{\pp}{\mathfrak{p}}
\newcommand{\kk}{\Bbbk}

\newcommand{\lint}{\Big\lfloor}
\newcommand{\rint}{\Big\rfloor}

\newcommand{\diam}{\operatorname{diam} }
\newcommand{\Min}{\operatorname{Min} }

\newcommand{\height}{\operatorname{height} }

\newcommand{\Proj}{\operatorname{Proj} }

\newcommand{\init}{\mathrm{in}_{\prec}}

\newcommand{\reg}{\operatorname{reg} }
\definecolor{mypink}{RGB}{215, 5, 234}

\begin{document}
\author{
Michela Di Marca
\\
\small Dip. di Matematica\\ \small Università di Genova\\
\small Genova, Italy, \\
\footnotesize \url{dimarca@dima.unige.it}
\and
Matteo Varbaro
\\
\small Dip. di Matematica\\ \small Università di Genova\\
\small Genova, Italy, \\
\footnotesize \url{varbaro@dima.unige.it}
 }

\title{On the diameter of an ideal}
\date{}
\maketitle 

\begin{abstract}
We begin the study of the notion of diameter of an ideal $I\subset S=\kk[x_1,\ldots ,x_n]$, an invariant measuring the distance between the minimal primes of $I$. We provide large classes of Hirsch ideals, i.e. ideals satisfying $\diam(I)\leq \height(I)$, such as: quadratic radical ideals such that $S/I$ is Gorenstein and $\height(I)\leq 4$, or ideals admitting a square-free complete intersection initial ideal.
\end{abstract}

\section*{Introduction}

The dual graph is a classical tool introduced in different contexts, as in algebraic geometry or in combinatorics, in order to study intersection patterns of algebraic curves, or combinatorial properties of simplicial complexes. More in general, it is possible to define the concept of dual graph for ideals in a Noetherian commutative ring. Under different names, this natural notion has already been studied by several authors, such as \cite{Ha,HH,BV,SS}.

\vspace{1mm}

Let $S=\kk[x_1,\ldots,x_n]$ be a polynomial ring in $n$ variables over a field $\kk$. For an ideal $I\subset S$, let $\Min(I)=\{\pp_1,\ldots,\pp_s\}$ be the set of minimal primes of $I$. The \textit{dual graph} $G(I)$ of $I$ is the graph $G=([s],E)$, where $[s]:=\{1,2,\ldots,s\}$ and 
\[\{i,j\}\in E \Leftrightarrow \height(\pp_i+\pp_j)=\height(I)+1.\]
In this paper, we are particularly interested in studying the diameter of $G(I)$, that we will name the \textit{diameter of $I$} and denote by $\diam(I)$, in the homogeneous case. In general the diameter of $I$ can be infinite. If $\diam(I)<\infty$, i.e. if $G(I)$ is connected, then $I$ must be height-unmixed, that is $\height(\pp)=\height(I) \ \forall \ \pp\in \Min(I)$. In this case, the number of vertices $s$ is at most the multiplicity $e(S/I)$ of $S/I$. Since a connected graph has diameter less than the number of vertices, then $\diam(I)< e(S/I)$. Synthetically, so:
\[\diam(I)<\infty\implies \diam(I) < e(S/I).\]

From a result of Hartshorne in \cite{Ha}, it follows that, if $I\subset S$ is homogeneous and $S/I$ is Cohen-Macaulay, then $\diam(I)<\infty$. By the above discussion, therefore, in this case $\diam(I)< e(S/I)$. This upper bound can be significantly improved in good situations.

In this spirit, we say that an ideal $I\subset S$ is \textit{Hirsch} if $\diam(I)\leq \height(I)$. One cannot expect more: an easy ideal such as $I=(x_1y_1,\ldots ,x_ny_n)\subset \kk[x_i,y_i:i=1,\ldots ,n]$, which is a square-free monomial complete intersection, satisfies $\diam(I)=\height(I)$.

The name comes from a conjecture made by Hirsch in $1957$. He conjectured that if a simplicial complex $\Delta$ is the boundary of a convex polytope, then its Stanley-Reisner ideal $I_{\Delta}$ is Hirsch. The conjecture has been disproved by Santos in $\cite{Sa}$. 

However, under some additional hypotheses, there are some positive answers. For example, Adiprasito and Benedetti proved Hirsch's conjecture when $\Delta$ is flag in \cite{AB}. More generally, they proved that if $I$ is monomial, quadratic, and $S/I$ satisfies Serre's condition $(S_2)$, then $I$ is Hirsch. Starting from this, in \cite{BV} Benedetti and the second author made the following general conjecture:

\begin{conjecture*}\label{conj}
Let $I\subset S$ be a homogeneous ideal generated in degree $2$. If $S/I$ is Cohen-Macaulay, then $I$ is Hirsch.
\end{conjecture*}

Beyond monomial ideals, the conjecture is known to be true in other situations. For example, recently Bolognini, Macchia and Strazzanti checked it for binomial edge ideals of bipartite graphs in \cite{BMS}. Motivated by the above conjecture, in this paper we provide large classes of Hirsch ideals:

\begin{theorem*}
A radical homogeneous ideal $I\subset S$ is Hirsch in the following cases:
\begin{compactitem}
\item[{\rm i)}] If $S/I$ is Gorenstein and $\reg S/I\leq 3$.
\item[{\rm ii)}] If $S/I$ is Gorenstein and $I$ is an ideal generated in degree $2$ of height $\leq 4$.
\item[{\rm iii)}] If $S/I$ is Gorenstein and $I$ is an ideal generated in degree $2$ of height $5$ which is not a complete intersection.
\item[{\rm iv)}] If $\kk=\QQ$ or $\kk=\ZZ/p\ZZ$ and there is a term order $\prec$ on $S$ such that $\init I$ is a square-free complete intersection.
\item[{\rm v)}] If $I$ is a complete intersection of height $2$ having one product of linear forms among its minimal generators.
\end{compactitem} 
\end{theorem*}

Point v) explains why many examples provided in \cite{BDV} of complete intersection ideals defining line arrangements in $\PP^3$ are Hirsch. In \cite[Example D]{BDV}, however, was provided a nonHirsch complete intersection line arrangement in $\PP^3$; in particular, in v) the word ``minimal'' is crucial (since any ideal defining a subspace arrangement obviously contains a product of linear forms).

\vspace{1mm}

The structure of the paper is as follows: After recalling some basic definitions and facts from commutative algebra, in Section 1 we discuss the connectedness properties of a weighted graph, introducing the notion of $(r,w)$-connectedness and studying its influence on the diameter. The achieved results will be used together with techniques from \cite{BBV} in Section 2 to show Theorem \ref{h-vec} and Corollary \ref{smallcodimension}, corresponding to points i), ii) and iii) of the Main Theorem. Using different methods, we also prove Proposition \ref{union}, corresponding to the point v). In Section 3, we discuss the relationship between the diameters of an ideal and its initial ideal (with respect to some term order). While in general there is almost no relation, inspired by the work of Knutson we prove Theorem \ref{comp} and Corollary \ref{cf}, which says that $\diam(I)\leq \diam(\init I)$ for certain Frobenius splitting ideals. As a consequence, we get Corollary \ref{ci}, corresponding to point iv) of the Main Theorem.

\vspace{1mm}

\textit{Acknowledgements}: The authors wish to thank Bruno Benedetti for sharing with them many interesting conversations about Hirsch's conjecture.

\subsection*{Preliminaries}
In our setting, $S$ will be the polynomial ring in $n$ variables, $S=\kk[x_1,\ldots,x_n]$, where $\kk$ is a field, and $I\subset S$ will be a homogeneous ideal of height $c$.
The \textit{Hilbert series} of $R=S/I$ is given by the rational series
\[H_{R}(t):=\sum_{i\geq 0}(\dim_{\kk} R_i)t^i=\frac{h(t)}{(1-t)^d},\]
where $d$ is the Krull dimension of $R$ and $h(t)=h_0+h_1t+\ldots+h_rt^r$ is a polynomial with integer coefficients. The $h$-vector of $R$ is then $h=(h_0,h_1,\ldots,h_r)$.
The \textit{multiplicity} of $R$ is $e(R):=h(1)$. We have the following additive formula:
\begin{equation}
e(R)=\sum_{\substack {\pp \in \Min(I) \\ \height(\pp)=c}}\lambda(R_{\pp})e(S/\pp),
\end{equation}
where $\lambda(-)$ means length of an $S$-module. If $I$ is radical, the formula above gives
\begin{equation}
e(R)=\sum_{\substack {\pp \in \Min(I) \\ \height(\pp)=c}}e(S/\pp).
\end{equation}
If $ 0 \rightarrow F_{p} \rightarrow \ldots\rightarrow F_1\rightarrow F_0\rightarrow R\rightarrow 0$ is the minimal graded free resolution of $R=S/I$ as $S$-module, then
the \textit{(Castelnuovo-Mumford) regularity} of $R$ is
\[\reg R=\min\{j : F_i \mbox{ is generated in degrees }\leq i+j \mbox{ for all }i\}.\]
When $I=(f_1,\ldots,f_c)$ is a complete intersection ideal with $d_i=\deg(f_i)$ for $i=1,\ldots, c$, the Koszul complex gives that the regularity of $R=S/I$ is:
\[\reg R=d_1+\ldots+d_c-c.\]
Throughout the paper, by a \textit{graph} $G=(V,E)$ we mean a simple graph with a finite set of vertices $V\neq \emptyset$ and set of edges $E$. Sometimes, if $|V|=s$, we will assume without loss of generality that $V=\{1,\ldots,s\}=:[s]$.
A \textit{path} in $G$ is an ordered sequence of distinct vertices in $V$, $(v_1,\ldots, v_m)$, such that $\{v_i,v_{i+1}\}$ is an edge of $G$ for all $i=1,\ldots,m-1$. The \textit{length} of such a path is $m-1$. 
Two vertices $u,v \in V$ are \textit{connected} if there exists a path $(v_1,\ldots,v_m)$ in $G$ connecting them, i.e. such that $v_1=u$ and $v_m=v$. The \textit{distance} between $u$ and $v$ is the maximum length of a path connecting them if they are connected, 0 if $u=v$, and $\infty$ otherwise. The \textit{diameter} of $G$, denoted by $\diam(G)$, is the maximum distance between two vertices in $V$.
A set of $k$ vertices $X=\{v_1,\ldots,v_k\}\subset V$ is said to \textit{disconnect} $G$, if $|V|\geq k+2$ and by removing $X$ from $G$ there are two disconnected vertices in $G$.
\begin{definition*}
Given a nonnegative integer $l$, a graph $G$ is said to be \textit{$l$-connected} if it has at least $l$ vertices and it cannot be disconnected with a set of $l-1$ or less vertices. 
\end{definition*}

Notice that a graph $G=(V,E)$ is connected if and only if it is 1-connected, if and only if $\diam(G)<\infty$. Furthermore, if it is $l$-connected then it is $l'$-connected whenever $l'\leq l$. The following is classical:

\begin{theorem**}[Menger]\label{M1}
$G$ is $l$-connected on at least $l+1$ vertices if and only if for every couple of vertices in $G$ there are at least $l$ pairwise disjoint paths connecting them.
\end{theorem**}

\begin{remark*}
By counting the number of vertices contained in the $l$ paths of Menger's theorem, we have the following inequality for an $l$-connected graph $G$ on $s$ vertices:
\begin{equation}\label{Mbound}
\diam(G)\leq \lint\frac{s-2}{l}\rint +1.
\end{equation} 
\end{remark*}

\section{Combinatorics of $(r,w)$-connected graphs}

In this section we discuss some bounds on the diameter for weighted graphs that satisfy certain connectedness properties. First we introduce a weighted version of the notion of $l$-connectedness.
 
Suppose that the graph $G=(V,E)$ is endowed with a \textit{weight function}
\[w:V\rightarrow \NN^+,\]
and let us denote $w_i:=w(i)$ the \textit{weight} of the vertex $i \in V$.

\begin{definition}
Given a graph $G=(V,E)$ and a weight function $w$ on $V$, we say that $G$ is \textit{$(r,w)$-connected} if the removal of a set of vertices having sum of their weights at most $r-1$ does not disconnect $G$.
\end{definition}

Note that if $G$ is $(r,w)$-connected, then it is also $(r',w)$-connected for all $r'\leq r$.

\begin{remark}
With every graph $G$, we can associate the weight function $\mathbb{I}$ that gives to all the vertices weight $1$. Hence, saying that $G$ is $l$-connected is equivalent to say that $G$ is $(l,\mathbb{I})$-connected. In general, if $G$ is $l$-connected, then $G$ is $(l,w)$-connected for every weight function $w$.
\end{remark}

\begin{proposition}\label{bound}
Let $G=(V,E)$ be an $(r,w)$-connected graph with $\sum_{i\in V}w_i=e$. Then
\[
\diam(G)\leq \lint \frac{e-2}{r}\rint +1.
\]
\end{proposition}
\proof
Let $x$ and $y$ be two vertices having maximum distance in $G$, and let $V_i$ be the set of vertices having distance $i$ from $x$, for $i=0,\ldots,\diam(G)$. 
Note that the $V_i's$ are pairwise disjoint and, by construction, there are no edges from $V_i$ to $V_j$ with $|i-j|\geq 2$.

Hence, for $i\in \{1,\ldots,\diam(G)-1\}$, the removal of the set $V_i$ disconnects $x$ and $y$; it means that the sum of the weights of the vertices in $V_i$ is at least $r$.
So we have
\[e\geq \sum_{i=1}^{\diam(G)-1}\sum_{j\in V_i}w_j+w_x+w_y \geq (\diam(G)-1)r+2.\]
Since the diameter is an integer number, we get the thesis.
\endproof

Observe that, for $w=\mathbb{I}$, this bound recovers Menger's one \eqref{Mbound}.
Furthermore, this bound is the best possible that taking into account only $r$ and $e$:
\begin{example}
Fix two integers, $e$ and $r$, and consider the following connected graph $G$ on $e$ vertices: fix two vertices $x$ and $y$, and distribute $d-2$ vertices in $\lint \frac{e-2}{r}\rint=:a$ sets $V_1,V_2,\ldots,V_a$ such that $e-2 \mod(r)$ of the $V_i$'s have cardinality $r+1$, and the remaining sets have cardinality $r$. Connect $x$ with all the vertices in $V_1$, $y$ with all the vertices in $V_a$, and each one of the vertices in $V_i$ with all the vertices in $V_{i+1}$, for all $i=1,2,\ldots,a-1$. By construction, the graph $G$ is $(r,\mathbb{I})$-connected and the number of $V_i$'s will be exactly the diameter of $G$ minus one, so
\[\diam(G)= a+1=\lint \frac{e-2}{r}\rint +1.\]
\end{example}

If one knows the weight function a priori, we will show that the bound on the diameter can be made more precise.
Given a graph $G$ on $s$ vertices with a weight function $w$, we can assume for our purposes that the vertices of $G$ are indexed in a non-decreasing order with respect to their weights: 
\[w_1\leq w_2 \leq \ldots \leq w_s.\]
Suppose that $G$ is $l$-connected and $(r,w)$-connected for some $l,r\geq 1$. For every $i$ such that $l+1\leq i\leq s$, let us define
\[\mathcal{A}_i:=\{1, 2,\ldots, k : k\geq i-1 \mbox{ and }\sum_{j=k}^{k-i+2}w_j \leq r-1\}.\]
In other words, $\mathcal{A}_i$ is a maximal set such that each of its subsets of cardinality $i-1$ has sum of the weights less than or equal to $r-1$. Furthermore, since $G$ is $(r,w)$-connected, we are sure that by removing a set of $i-1$ or less vertices chosen from $\mathcal{A}_i$, we do not disconnect $G$.
Let $h:=\max\{i\ :\ \mathcal{A}_i\neq \emptyset \}$, and let $b_i:=|V \setminus \mathcal{A}_i|$ for $i=l+1,\ldots,h$. We have
\[\mathcal{A}_h\subseteq \mathcal{A}_{h-1}\subseteq \ldots \subseteq \mathcal{A}_{l+1}\ \mbox{  and  }\ b_h\geq b_{h-1}\geq \ldots\geq b_{l+1}.\]

\begin{proposition}\label{cor}
Let $G$ be a graph on $s$ vertices that is $l$-connected and $(r,w)$-connected for some $r\geq l\geq 1$. With the notation above, we have 
\[
\diam(G) \leq \lint \frac{s-2+\sum_{i=l+1}^h b_i}{h}\rint +1.
\]
\end{proposition}
\proof
Let $x$ and $y$ be two external points of a diameter of $G$ and let $V_i$ be the subset of $V$ of vertices having distance $i$ from $x$, $i=0,\ldots, \diam(G)$. 
As in the proof of Proposition \ref{bound}, for $i\in \{1,\ldots,\diam(G)-1\}$ the removal of a set $V_i$ must disconnect $x$ and $y$. Fix a certain $i\in \{1,\ldots,\diam(G)-1\}$; since $G$ is $l$-connected, there have to be at least $l$ distinct vertices $y_{i,1},\ldots,y_{i,l}$ in $V_i$.

If there are not other vertices in $V_i$, it means that $\{y_{i,1},\ldots,y_{i,l}\} \nsubseteq \mathcal{A}_{l+1}$. We can suppose $y_{i,l}\in V\setminus \mathcal{A}_{l+1}$, and for every $j=1,\ldots,h-l$ we define 
$y_{i,l+j}:=y_{i,l}\in V\setminus \mathcal{A}_{l+1}\subseteq V\setminus \mathcal{A}_{l+j}.$
If instead there is another vertex $y_{i,l+1}$ in $V_i$, we have that 
$|\{y_{i,l},\ldots,y_{i,l+1}\}|=l+1$
and we can go on: if there are no other vertices in $V_i$, it means that there is at least a vertex in $V\setminus \mathcal{A}_{l+2}$, we can suppose this is $y_{i,l+1}$, and for every $j=2,\ldots,h-l$, we can define $y_{i,l+j}:=y_{i,l+1}\in V\setminus \mathcal{A}_{l+2}\subseteq V\setminus \mathcal{A}_{l+j}$;
otherwise, there is another vertex $y_{i,l+2}\in V_i$ and $|\{y_{1,l},\ldots,y_{1,l+2}\}|=l+2$.
By continuing in such a way, for every $i$, we can construct a sequence of vertices \[y_{i,1},\ \ldots,\   y_{i,l},\ y_{i,l+1}, \ldots,\ y_{i,h}\] such that:
\[|\{y_{i,1},\ldots,y_{i,l}\}|=l \mbox{, and, for } j=1,\ldots,h-l,\ |\{y_{i,1},\ldots,y_{i,l+j}\}|=l+j\ \mbox{  or  } \ y_{i,l+j}\in V\setminus \mathcal{A}_{l+j}.\]
Consider the set all the vertices $y_{i,j}$, for $i=1,\ldots,\diam(G)-1$, $j=1,\ldots, h$, with possible repetitions; since a vertex $y_{i,j}$ can be repeated only if contained in $V\setminus \mathcal{A}_j$, we can bound the number of distinct vertices in $\cup_i V_i$ by using the $b_j$'s. We get the following inequality:
\[s \geq (\diam(G)-1)h-b_{l+1}-b_{l+2}-\ldots-b_h+2.\]
The thesis follows by considering that $\diam(G)$ is integer.
\endproof

In some situations, if one knows the weight function, the bound above is better than the one in Proposition \ref{bound}:

\begin{example}
Consider the graph $G$ as the star graph having center in a vertex with weight $5$ and six rays ending in vertices with weights $2$. Then $G$ is connected, $(5,w)$-connected (but not $(6,w)$-connected). We have:
\[w_1=2\leq w_2=2\leq w_3=2\leq w_4=2\leq w_5=2\leq w_6=2\leq w_7=5.\]
\[ \!\!
\begin{array}{lll}
\mathcal{A}_1=V & & l=1\\
\mathcal{A}_2=\{1,2,3,4,5,6\} &   & b_2=1\\
\mathcal{A}_3=\{1,2,3,4,5,6\} &  & b_3=1\\
\mathcal{A}_4=\emptyset && h=3. 
\end{array}
\]

Hence, the bound coming from Proposition \ref{cor} is $\diam(G)\leq 3$, while $e=17$ and so the bound coming from Proposition \ref{bound} is $\diam(G)\leq 4$.
Note also that the graph is connected but not $2$-connected. Hence, Menger's bound \eqref{Mbound} is $\diam(G)\leq 6$.

In this case, however, none of the three bounds is optimal since $\diam(G)=2$.
This depends in part on the fact that $\mathcal{A}_4$ does not coincide with the set of maximum cardinality such that the removal of $3$ vertices from it do not disconnect $G$. If we chose the following sets:
\[\mathcal{A'}_i\in \{A \subseteq V:i-1 \mbox{ or less vertices from $A$ do not disconnect $G$}, A \mbox{ maximal}\},\]
for $i=2,\ldots,|V|$, we could replace them to the $\mathcal{A}_i$'s in the proof of Proposition \ref{cor}. In this case, we obtain $h=7$ and 
\[\mathcal{A'}_2=\mathcal{A'}_3=\mathcal{A'}_4=\mathcal{A'}_5=\mathcal{A'}_6=\mathcal{A'}_7=\{1,2,3,4,5,6\}.\] 
Then, we get the bound $\diam(G)\leq 2$, which is sharp. Of course, without knowing the structure of the graph, it would have been impossible to compute the $\mathcal{A'}_i$'s only by knowing the weight function and $r$.
\end{example}

\section{Some classes of Hirsch ideals}

In this section we will apply the results of the previous section in order to bound the diameters of radical ideals defining Gorenstein $\kk$-algebras. For this section we assume that $\kk$ is infinite.

The following Theorem gives a bound on $\diam(I)$ that depends on the multiplicity and the Castelnuovo-Mumford regularity of $S/I$. 

\begin{theorem}\label{deg}
Let $S/I$ be reduced and Gorenstein with $\reg S/I=r$, $e(S/I)=e$. Then
\[ \diam(I)\leq \lint\frac{e-2}{r}\rint +1.\]
More precisely, if $\Min(I)=\{\pp_1,\ldots,\pp_s\}$ and we endow $G(I)$ with the weight function $w$ defined, for $i=1,\ldots ,s$, as $w_i=e(S/\pp_i)$, then $G(I)$ is $(r,w)$-connected.
\end{theorem}
\proof
Following the proofs of \cite[Theorem 4.3 and Corollary 4.4]{BBV}, the removal of a set of vertices such that the sum of the degrees of the corresponding minimal primes is at most $r-1$ does not disconnect $G(I)$.  Therefore $G(I)$ is $(r,w)$-connected, and since $e=\sum_{i=1}^sw_i$ Proposition \ref{bound} yields the diameter bound.
\endproof

\begin{remark}
Note that, as we mentioned in the introduction, if $S/I$ is Cohen-Macaulay, then $\diam(I)\leq e(S/I)-1$.
Theorem \ref{deg} improves substantially this bound for $S/I$ reduced and Gorenstein.
\end{remark}
 
Here one could use Proposition \ref{cor} to bound $\diam(I)$ in terms of the $\mathcal{A}_i$'s, the $b_i's$ and $h$ computed for $G(I)$ with weight function as in Theorem \ref{deg}. The following example shows the difference of accuracy of the bounds derived from Propositions \ref{bound} and \ref{cor}.

\begin{example}
Let $S=\QQ[x,y,z,t,w]$ and consider the following ideal:
\[I=((x+w)(x-w)+z^2,(x+t)(y+2t)+z(x+y),zt).\]
The ideal $I$ is a radical complete intersection ideal and $\reg S/I=3$. With respect to the multiplicity weight function $w$, $G(I)$ is $(3,w)$-connected by Theorem \ref{deg}. One can check with \textit{Macaulay2} \cite{M2} that the minimal primes of $I$ are:
\[ \!\!
\begin{array}{lll}
\pp_1=(z,x+w,x+t), & \: \pp_3=(z,x-w,x+t),&  \: \pp_5=(t,x^2+z^2-w^2,xy+z(x+y)).\\
\pp_2=(z,x-w,y+2t), & \: \pp_4=(z,x+w,y+2t),
\end{array}
\]

The multiplicities of the minimal primes are, respectively:
\[w_1=w_2=w_3=w_4=1, w_5=4.\]
The dual graph of $G(I)$ is the complete graph on 5 vertices minus two disjoint edges ($\{2,5\}$ and $\{3,4\}$), so $\diam(I)=2$. In this case, $h=3, b_2=1, b_3=1$. Hence, the bound given by Proposition \ref{cor} is sharp, contrary to the one in Theorem \ref{deg}.

\end{example}

\begin{theorem}\label{h-vec}
Assume that $S/I$ is reduced and Gorenstein. 
If $\reg S/I\leq 3$, then $I$ is Hirsch.
\end{theorem}
\proof

We can assume that $I$ does not contain linear forms. Since the $h$-vector is symmetric and its length is equal to the regularity plus 1, the possible $h$-vectors for $S/I$ as in the assumptions are the following:
\[ \!\!
\begin{array}{lll}
h=(1,1) & & \mbox{if } \reg S/I=1;\\
h=(1,\height(I),1) &  & \mbox{if } \reg S/I=2;\\
h=(1,\height(I),\height(I),1) & & \mbox{if } \reg S/I=3.
\end{array}
\]
Because the multiplicity of $S/I$ equals $h(1)$, by Theorem \ref{deg} we get
\[\diam(I)\leq \max\Big\{1, \lint \frac{\height(I)}{2}\rint +1, \lint \frac{2\height(I)}{3}\rint +1\Big\}\leq \height (I).\]

\endproof

\begin{lemma}\label{l:ci}
Let $I, J$ be two homogeneous ideals of $S$, $I\subseteq J$. Suppose that $S/J$ is Cohen-Macaulay, $S/I$ is Gorenstein and $\height(I)=\height(J)$. Then $\reg S/I\geq \reg S/J$ and
\[\reg S/I= \reg S/J\iff I=J.\]
\end{lemma}
\proof
Let $d:=\dim(S/I)=\dim(S/J)$. Being $S/I$ and $S/J$ Cohen-Macaulay and $\kk$ infinite, it is possible to find $\underline{l}=l_1,\ldots,l_d \in S_1$ that form a $S/I$-regular sequence that is also a $S/J$-regular sequence. Let $A=S/(I+(\underline{l}))$ and $B=S/(J+(\underline{l}))$. There is a natural surjective map of graded Artinian $k$-algebras
\[\phi: A \twoheadrightarrow B.\]
Let us call $r=\reg S/I=\max\{s:A_s\neq 0\}$ and $t=\reg S/J=\max\{s:B_s\neq 0\}$. 

From the fact that $\phi$ is graded, we can see that $r\geq t$. Since $S/I$ is Gorenstein, the $\kk$-vector space $A_r$ is generated by a nonzero element $f$. If $r=t$, $B_r\neq 0$, and hence $\phi(f)\neq 0$. We will show that $\phi$ is then an isomorphism. Suppose that for some $j<r$ there is a non-zero element $g\in A_j$ such that $\phi(g)=0$. Since $g \notin\ A_r=(0:\mm)$,

we can find an element $h\in A_{r-j}$ such that $gh=f$. But $0\neq \phi(f)=\phi(gh)=\phi(g)\phi(h)$, which is a contradiction against the fact that $\phi(g)=0$. Hence $\phi$ is injective and $I=J$.
\endproof

\begin{corollary}\label{smallcodimension}
The conjecture in the introduction is true when $S/I$ is reduced and one of the following conditions holds:
\begin{compactitem}
\item[{\rm 1)}] $S/I$ is Gorenstein and $\height(I)\leq 4$;
\item[{\rm 2)}] $S/I$ is Gorenstein but not a complete intersection and $\height(I)=5$.
\end{compactitem} 
\end{corollary}
\proof
If $I$ is a quadratic complete intersection and $\height(I)\leq 4$ we have that $e(S/I)=2^{\height(I)}$, $\reg S/I=\height(I)$, so Theorem \ref{deg} gives the Hirsch bound. If $I$ is not a complete intersection, we can find a complete intersection of quadrics $J\subsetneq I$ with $\height(I)=\height(J)$. By Lemma \ref{l:ci}, $\reg S/I<\reg S/J=\height(J)\leq 4$. Therefore $\diam(I)\leq \height(I)$ by Theorem \ref{h-vec}. This proves point 1).

For point 2), consider a complete intersection of quadrics $J\subsetneq I$ with $\height(I)=\height(J)$. By Lemma \ref{l:ci}, the $h$-vector of the Gorenstein ring $S/I$ has length less than $\reg S/J=5$. In the case that $\reg S/I=4$ (that is the only case not covered by Theorem \ref{h-vec}), this is of kind
\[h=(1,5,h_2,5,1),\]
where $h_2<10$, as otherwise the number of quadrics in $S/I$ would be the same as in the complete intersection case. Hence, by applying Theorem \ref{deg}, $\diam(I)\leq 5$. 
\endproof

\begin{remark}
When $I$ is a radical quadratic complete intersection of height 5, Theorem \ref{deg} only yields $\diam(I)\leq 7$. It would be very interesting to know whether $\diam(I)=7$, or even $6$, is actually possible. 
\end{remark}

The following proposition proves that a complete intersection curve in $\mathbb{P}^3$ that is contained in a union of planes is Hirsch.

\begin{proposition}\label{union}
Let $I=(f,g) \subset S=\kk[x_1,x_2,x_3,x_4]$ be a complete intersection ideal where $g=\prod_{i=1}^e l_i$ for some $l_i \in S_1$ and $\mathrm{MCD}(f,g)=1$. Then $I$ is Hirsch.
\end{proposition}
\proof
Consider $X=\Proj(S/(f)), Y=\Proj(S/(g))$, and let $C_1,\ldots, C_n$ be the primary components of $Z=X\cap Y$. Note that $\sum_{i=1}^n \deg(C_i)=de$, where $d=\deg(f)$. By assumption, $Y$ is a union of $e$ planes, say $H_1,\ldots,H_e$; it is harmless to assume that $Y$ is reduced, so that we can assume $H_i\neq H_j$ if $i\neq j$. Observe that each plane contains a set of curves of $Z$ whose sum of the degrees is exactly $d$, as otherwise the intersection between a plane and a surface of degree $d$ would have degree more than $d$.

Now consider two curves in $Z$, $C_a$ and $C_b$, and let us show that their distance in the dual graph is at most $2$. Suppose that $C_a\cap C_b=\emptyset$, and fix two planes of $Y$: $H_k\supset C_a$ and $H_m\supset C_b$. If the line $R_{km}:=H_k\cap H_m$ is a component of $Z$, we are finished. Otherwise, suppose that $R_{km}\neq C_i$ (as sets) $\forall \ i=1,\ldots,n$. Since 
\[R_{km}\cap Z=(R_{km}\cap H_k)\cap Z=R_{km}\cap (H_k\cap X)=R_{km}\cap (\cup_{C_i\subset Z\cap H_k}C_i),\]
$R_{km}\cap Z$ is the intersection of a line with a plane curve of degree $d$, and hence there are exactly $d$ points $P_1,\ldots, P_d$ on it (counted with multiplicity). 

Since the point $P:=C_b\cap R_{km}$ is a point of $R_{km}\cap Z$, it must coincide with one of the $P_i$'s. Hence, there exists a curve $C_p \subset H_k\cap X$ such that $C_p\cap R_{km}=P$. Then $C_p\cap C_b=\{P\}\neq \emptyset$ and $C_p\cap C_a \neq \emptyset$ because $C_a, C_p \subset H_k$. So the distance between $C_a$ and $C_b$ in $G(I)$ is 2.
\endproof

\begin{remark}
Schl\"afli's double six is the smallest example of a graph representing a complete intersection of lines in $\mathbb{P}^3$ that is not Hirsch (see \cite[Example D]{BDV}).
In fact, for a complete intersect of kind $(d,e)$,
\[\diam(G)\leq \min\{d,e\} \quad (\cite[\mbox{Lemma 1.6}]{BDV}),\]
and the dual graph of a reduced complete intersection of two cubics satisfies $\diam(G)\leq 2$ by Theorem \ref{deg}.
\end{remark}

\section{Diameters of initial ideals}

The aim of this section is to understand if there exists, and in case what it is, a relation between the dual graph of an ideal and the dual graph of its initial ideal with respect to some term order $\prec$ on the polynomial ring $S$. We can, and will, assume that $x_i>x_j$ whenever $i<j$. All the results in this section hold also in the non-homogeneous setting. In general, we have the following result, which is a consequence of \cite{Va}:

\begin{theorem}
If $I\subset S$ is an ideal such that $\diam(I)<\infty$, then $\diam(\init I)<\infty$.
\end{theorem} 
\begin{proof}
For any ideal $J\subset S$, the dual graph $G(J)$ is connected if and only if the connectivity dimension of $\mathrm{Spec}(S/J)$ is greater than or equal to $\dim(S/J)-1$ by \cite[Remark 1.1]{Va}. So the thesis immediately follows by \cite[Theorem 2.5]{Va}.
\end{proof}

The viceversa does not hold:
\begin{example}
Let $I=(x_1+x_2,x_3)\cap (x_2,x_3+x_4)\subseteq \kk[x_1,x_2,x_3,x_4]$. Then $G(I)$ consists of two disconnected vertices. Choosing as $\prec$ the lexicographic order, however, $\init I=(x_1x_2,x_1x_3,x_2x_3,x_3^2)$ and $\Min(\init I)=\{(x_1,x_3),(x_2,x_3)\}$, so the dual graph of $\init I$ is a path of length one. Hence $\diam(\init I)<\infty$, yet $\diam(I)=\infty$.
\end{example}

In general, the gap between $\diam(\init I)$ and $\diam(I)$ can be arbitrarily large in both directions:

\begin{example}
Let $I$ be an ideal with diameter $m$, for example consider in $\kk[x_1,\ldots,x_{m+2}]$ the intersection of the $m+1$ prime ideals 
\[(x_1,x_2),(x_2,x_3),\ldots, (x_m,x_{m+1}), (x_{m+1},x_{m+2}).\]
Let $g \in \mathrm{GL}_{m+2}(\kk)$ be a generic coordinate change, and let $J=g(I)$. Then $J$ has the same dual graph of $I$. However, if $\prec$ is the degree reverse lexicographic order $\init J$ has only one minimal prime (see \cite[Section 15.9]{Ei}). So 
\[m=\diam(J)\gg\diam(\init J)=0.\] 
\end{example}

\begin{example}
Let $X$ be the following matrix of variables over some field $\kk$:
\[X=\left( \begin{array}{cccc}
x_1 & x_2 & \ldots & x_m \\
x_{m+1} & x_{m+2} & \ldots & x_{2m}  \end{array} \right).\]
Let $I \subset \kk[x_1,\ldots,x_{2m}]$ be the ideal of $2$-minors of $X$. It is well known that $I$ is a prime ideal. Furthermore, for any term order $\prec$ the $2$-minors form a Gr\"obner basis (\cite[Theorem 7.2]{SZ}). By choosing, for example, the lexicographic term order, hence we have
\[\init I=(x_ix_{m+k}\ |\ i+1\leq k \leq m, \ i=1,\ldots, m-1).\] 

One can easily check that the primary decomposition of $\init I$ is
\[\init I=\bigcap_{k=1}^m (x_1,\ldots,x_{k-1},x_{m+k+1},\ldots,x_{2m}).\]
By denoting $\pp_k:=(x_1,\ldots,x_{k-1},x_{m+k+1},\ldots,x_{2m})$, we have $G(\init I)=([m],E)$ where
\[\{i,j\}\in E \Leftrightarrow |i-j|=1,\quad i,j \in [m].\]
It follows that $G(\init I)$ is a path, hence
\[0=\diam(I)\ll\diam(\init I)=m-1.\]

\end{example}

However, under some additional assumptions the diameter of the initial ideal bounds from above the diameter of the ideal. For our next result we need the following simple fact (for the proof see for example \cite[Corollary 2]{Kn}).

\begin{lemma}\label{Knu} Let $A$ be a finite set of ideals in $S$ and $\prec$ a term order such that 
$\init (\cap _{I \in A'}I)$ is radical for all $ A'\subseteq A$. 
Then 
\[\init\left(\sum_{I \in A'} I\right)=\sum_{I \in A'}\init I \mbox{ \ and \ } \init\left(\bigcap _{I \in A'}I\right)=\bigcap _{I \in A'}\init I \ \ \ \ \ \forall \ A'\subseteq A.\]
\end{lemma}

\begin{theorem}\label{comp}
Let $I\subset S$ be a radical ideal with primary decomposition $I=\mathfrak{p}_1\cap \ldots \cap \mathfrak{p}_s$, and $\prec$ a term order. If $\init(\cap_{i\in A}\mathfrak{p}_i)$ is radical for all $A\subseteq [s]$, then 
\[\diam(I)\leq \diam(\init I).\]
\end{theorem}
\proof
By using Lemma \ref{Knu}, we have that 
\[\init I=\init \pp_1\cap \ldots \cap \init \pp_s.\]

Assume that, for $i=1,\ldots ,s$, the primary decomposition of $\init \pp_i$ is $\init \pp_i=\qq_{i,1}\cap \ldots \cap \qq_{i,s_i}$. 

We claim now that if $\{\pp_i,\pp_j\}$ is not an edge of $G(I)$, then $\{\qq_{i,h}, \qq_{j,k}\}$ is not an edge of $G(\init I)$ for every $h$ and $k$. 

Let $c=\height(I)$: If $\{\pp_i,\pp_j\}$ is not an edge of $G(I)$, then $\height(\pp_i+\pp_j)>c+1$, and so also $\height(\init(\pp_i+\pp_j))>c+1$. Using again Lemma \ref{Knu}, $\init(\pp_i+\pp_j)=\init \pp_i+\init \pp_j$. Since $\qq_{i,h}\supseteq \init \pp_i$ and $\qq_{j,k}\supseteq \init \pp_j$, we have that $\qq_{i,h}+\qq_{j,k}\supseteq \init \pp_i+\init \pp_j$ and hence
\[\height(\qq_{i,h}+\qq_{j,k})\geq \height(\init \pp_i+\init \pp_j)=\height(\init(\pp_i+\pp_j))>c+1.\] 

So, by definition, $\{\qq_{i_h}, \qq_{j_k}\}$ is not an edge of $G(\init I)$. If we take $\pp_i$ and $\pp_j$ as the two extreme vertices of a diameter of $G(I)$, it follows that $\qq_{i_h}$ and $\qq_{j_k}$ are two vertices of $G(\init I)$ having distance at least $\diam(I)$, and we get the thesis.
\endproof

In view of the above theorem, we feel the following is a natural question to ask:

\begin{question}
If $I\subset S$ and $\prec$ are such that $\init I$ is square-free, is it true that 
\[\diam(I)\leq \diam(\init I) \ ?\]
\end{question}

The assumption of Theorem \ref{comp} is not so frequent, however there is a natural class of ideals satisfying it. From now on either $\kk=\QQ$ or $\kk=\ZZ/p\ZZ$. Let $f\in S=\kk[x_1,\ldots,x_n]$, and consider the smallest set $\C_f$ of ideals of $S$ satisfying the following:
\begin{compactenum}
\item $(f)\in\C_f$;
\item If $I\in\C_f$, then $I:J\in\C_f$ for any $J\subseteq S$;
\item If $I,J\in\C_f$, then $I+J\in\C_f$ and $I\cap J\in\C_f$.
\end{compactenum} 

\begin{corollary}\label{cf}
Let $f\in S$ and $\prec$ a term order on $S$ such that $\init f$ is a square-free monomial. Then
\[\diam(I)\leq \diam(\init I) \ \ \ \forall \ I\in\C_f.\]
\end{corollary}
\begin{proof}
If $\kk=\ZZ/p\ZZ$, then the thesis follows from \cite[Theorem 2]{Kn} and Theorem \ref{comp}, whereas if $\kk=\QQ$ it follows from \cite[Theorem 4]{Kn} and Theorem \ref{comp}.
\end{proof}

The following is an immediate consequence:

\begin{corollary}\label{ci}
If $I\subset S$ is such that $\init I$ is a square-free complete intersection for some term order $\prec$, then $I$ is Hirsch.
\end{corollary}
\begin{proof}
Since $\init I$ is a complete intersection, then $I$ is a complete intersection as well. By the assumptions, if $c=\height(I)$, we can choose $c$ polynomials $f_1,\ldots ,f_c\in S$ such that $I=(f_1,\ldots ,f_c)$ and $\init I=(\init f_1,\ldots ,\init f_c)$. Therefore $I\in \C_{f_1\cdots f_c}$, so Corollary \ref{cf} yields $\diam(I)\leq \diam(\init I)$.
Furthermore, $\init I$ is a complete intersection, and it is not difficult to see that monomial complete intersections are Hirsch.
\end{proof}

\begin{remark}
An equivalent formulation of the above corollary is the following: If $f_1,\ldots ,f_c\in S$ are such that $\init f_1\cdots \init f_c$ is a product of distinct variables, then $(f_1,\ldots ,f_c)$ is Hirsch. 

\end{remark}

\begin{example}
Let $X=(x_{ij})$ be an $m\times n$ matrix of indeterminates over $\kk$, and let $f_1,\ldots,f_{n-m+1}$ be the \textit{diagonal} $m$-minors: specifically, $f_i$ means the $m$-minor insisting on the columns from the $i$th to the $(i+m-1)$th. By choosing the lexicographic term order we are in the situation of Corollary \ref{ci}, so the ideal $(f_1,\ldots,f_{n-m+1})$ is Hirsch.
\end{example}

\end{document}